\documentclass{scrartcl}
\usepackage{graphicx} 
\usepackage{cite}
\usepackage{algpseudocode}
\usepackage{enumerate}
\usepackage{amssymb}
\usepackage{amsmath}
\usepackage{amsthm}
\usepackage[english]{babel}
\addto\captionsenglish{}

\newtheorem{lemma}{Lemma}

\begin{document}
\title{\LARGE \bf Fast AC Power Flow Optimization using Difference of Convex Functions Programming}

\author{ 
  Sandro Merkli\thanks{
    Automatic Control Lab,
    ETH Zurich, Physikstrasse 3, 8092 Zurich
    \tt{ \{merkli,smith,juanj\}@control.ee.ethz.ch}}
  \thanks{
    Inspire-IfA, Inspire AG, Technoparkstrasse 1, 8005 Zurich. 
    \tt{\{merkli,domahidi\}@inspire.ethz.ch}}
    ,
  Alexander Domahidi\thanks{
    embotech GmbH, Physikstrasse 3, ETL K10.1, 8092 Zurich,
  \tt{\{jerez,domahidi\}@embotech.com}} \footnotemark[2] , 
  Juan Jerez\footnotemark[1] \footnotemark[3]\\
  Manfred Morari\footnotemark[1],
  Roy S.\ Smith\footnotemark[1]\\
}

\def \C { \mathbb C }
\def \R { \mathbb R }
\def \diag { \operatorname{diag} }
\def \minim { \operatorname*{minimize} }
\def \maxim { \operatorname*{maximize} }
\def \st { \operatorname*{subject\ to} }
\def \real { \operatorname{Re} }
\def \imag { \operatorname{Im} }
\def \eig { \operatorname{eig} }
\def \bmb { \begin{bmatrix} }
\def \bme { \end{bmatrix} }
\newcommand{\cve}[1] {#1}
\newcommand{\jcve}[1] {\bar{#1}}

\maketitle
\thispagestyle{empty}
\pagestyle{empty}
\begin{abstract}
%
%
An effective means for analyzing the impact of novel operating schemes on power
systems is time domain simulation, for example for investigating
optimization-based curtailment of renewables to alleviate voltage violations.
Traditionally, interior-point methods are used for solving the non-convex AC
optimal power flow (OPF) problems arising in this type of simulation.
This paper presents an alternative algorithm that better suits the simulation
framework, because it can more effectively be warm-started, has linear 
computational and memory complexity in the problem size per iteration and
globally converges to Karush-Kuhn-Tucker (KKT) points with a linear rate if
they exist.
The algorithm exploits a difference-of-convex-functions reformulation of the
OPF problem, which can be performed effectively.
Numerical results are presented comparing the method to state-of-the-art
OPF solver implementations in MATPOWER, leading to significant speedups
compared to the latter.
\end{abstract}

\section{Introduction}
The amount of renewable energy sources (RES) in distribution systems is
steadily increasing~\cite{ren212015global}. Due to their volatility and limited
predictability, they are posing new challenges to power system operation and
planning. 
%
%
%
A prominently observed consequence of the increase in renewable power in-feeds
are local voltage limit violations~\cite{Ayres2010}. Traditionally, the remedy 
for these violations required expensive line capacity extensions. Recent studies
have shown that such extensions could be reduced by a shift in operational
paradigms from rule-based to optimization-based approaches, see for
example~\cite{Warrington2014,Ulbig2012,Vrettos2013}.

Since it is non-trivial to predict the impact of such shifts in operational
paradigms on power systems, time-domain simulations provide valuable
insight~\cite{Ulbig2012}. System-wide simulations over extended periods of time
can demonstrate seasonal impacts and yield statistical data. This data provides
a more in-depth view than worst-case snapshot studies, which are the current
industrial practice.  While the latter only provides information on violation
severity, the former also gives a sense of how often they occur.
However, if the impact of optimization-based approaches is to
be simulated over such long periods of time and for different scenarios, a large
number of optimization problems need to be solved. In the case of dispatch
optimization, sampling times for the control are on the order of 15 minutes.
This means that proposed optimization problems can typically be solved fast
enough for on-line operation using state-of-the-art software such as
MATPOWER~\cite{Zimmerman2011}. However, in simulations, solving the
optimization problems is the most computationally expensive task. Therefore,
efficient numerical methods are essential for performing simulations in a
practical time frame. 

In many cases, the problems proposed in optimization-based operation schemes
are related to a class of problems collectively referred to as optimal power
flow (OPF) problems. An extensive amount of literature exists on solving such
problems and a recent survey is given
in~\cite{Frank2012surveyI,Frank2012surveyII}. However, due to the non-convexity
and large scale of the problem, it remains an active research topic. In fact,
the non-convexity makes the problem computationally intractable to solve to
global optimality in general. However, critical points can in most
cases be found efficiently if they exist, for example using sequential quadratic
programming (SQP)~\cite{Robinson1972} and an initial guess that is close to the
critical point, or with the difference-of-convex-functions method used in this
paper~\cite{LeThi2014}.  The most popular approaches to solving AC OPF problems
are interior point methods~\cite{Torres2001} and sequential convex
approximation methods~\cite{Alsac1990,Chang1990}. While the former are
numerically robust and well-studied, the latter tend to be faster according
to~\cite{Frank2012surveyI}. There are two main sequential convex approximation
approaches: Sequential linear programming (SLP) and SQP. These schemes
approximate the original problem iteratively with convex linear and quadratic
programs, respectively. Most implementations of these approaches use
conventional power flow computations between their iterations to restore
feasibility of the Kirchhoff equations.  In general, SLP/SQP methods require
extensions to become globally convergent, which reduces their
performance~\cite{Frank2012surveyI}. A recently developed alternative approach
is to solve a convex semidefinite programming (SDP) relaxation of the
problem~\cite{Lavaei2012,Molzahn2013a}. The optimal value of this relaxation is
either the globally optimal value of the non-convex problem or in the worst
case only a lower bound on the latter. 


The reformulation presented in this paper allows for the solution of AC optimal
power flow problems using difference of convex functions programming. This
higher order approximation is tighter than the one made in SLP/SQP methods. In
comparison with SLP, the method has no issues of unboundedness of the
relaxations and is globally convergent without extensions. The presented method
operates entirely in the voltage space, satisfying the Kirchhoff equations by
design and thereby eliminating the need for conventional power flow
computations.
In comparison with the SDP relaxation, it converges to critical points at a
lower computational cost than the former, especially when warm-started. Also,
the SDP relaxation provides only a lower bound on the objective in the worst
case, which does not provide a feasible point. In contrast, the proposed method
always converges to a critical point if one exists. While certifying local
optimality of these points is not straightforward, experiments show that they
represent acceptable solutions. This statement will be quantified in the
numerical results section.

In this work, we present the application of our method to a specific example
of an optimization-based operation scheme designed to reduce RES curtailments.
In this example, the distribution system operator (DSO) is tasked with keeping
the system stable and within the allowed operating conditions. Normally, 
the DSO is not operated for financial gain and its actions are bound to
regulations, for example the EEG in Germany~\cite{eeg2014} or the European
equivalent ENTSO~\cite{entsoe}. 
The range of actions the DSO can take includes
adjusting setpoints of generators and curtailing renewable energy sources.
Approaches for finding such points
currently used in practice are usually rule-based. Such rules involve a
significant amount of tuning and rarely come with mathematical guarantees.
Additionally, costs for adjustments can only indirectly be taken into account. 

While the rest of the paper is developed with this specific example in mind,
the theory applies to a wide range of problems involving similar constraints,
including standard economic dispatch. In particular, any AC power flow 
optimizations can make use of the decomposition technique presented here.

\subsection{Summary of contribution}
In this paper, we present a novel method for solving a class of optimal power
flow problems that is particularly suited for time-domain system simulations.
\begin{enumerate}[(i)]
  \item \emph{Formulation}: We propose an OPF-like optimization problem to
    reduce curtailment in distribution grid operation. While the constraints
    are similar to economic dispatch, the cost function is formulated
    specifically to represent the cost faced by the system operator.
  \item \emph{Reformulation into a difference-of-convex functions problem}: 
    We give an efficient method for transforming the given OPF problem
    into a difference of convex functions problem. Its computational complexity
    is linear in the problem size. This reformulation preserves the
    sparsity of the problem while at the same time leading to the sequential
    convex relaxations being as close to the original non-convex problem as
    possible within the DC programming framework.
  \item \emph{Efficient solution of convex subproblems}: The difference
    of convex functions approach solves the non-convex problem using a 
    series of convex approximations, in this case second-order cone programs
    (SOCPs) that can be reformulated as convex quadratically constrained
    linear problems (QCLPs). We present an approach using accelerated dual 
    projected gradient methods to solving these QCLPs that exploits their
    structure. This leads to the complexity of all iteration computations as
    well as the required amount of memory growing linearly with the problem
    size.
\end{enumerate}

\subsection{Outline}
The rest of this paper is structured as follows: Section~\ref{sec:prelim}
introduces some preliminaries. In Section~\ref{sec:reform}, we present a
reformulation of the optimization problem such that the difference-of-convex-functions 
method is applicable. Section~\ref{sec:effinner} outlines an efficient
method to solve the convex inner problems arising in the proposed algorithm. In
Section~\ref{sec:numres}, numerical results are presented and discussed.  Final
conclusions are presented in Section~\ref{sec:conclusion}.

\section{Preliminaries}
\label{sec:prelim}
This section outlines both the model of the power system as well as the
optimization-based control strategy we propose. Basic notation is introduced,
assumptions are clarified and current operational practice is described.

\subsection{Notation}
The power grid is modeled as an undirected graph with $M$ vertices and $L$
edges. Vertices model buses, while edges model power lines. Each line (say, from
bus $j$ to bus $l$) has admittance $y_{jl} \in \C$. Each bus $j$ has an
associated voltage $v_j \in \C$ and power in-feed $s_j \in \C$, where
$\real(s_j)$ denotes active and $\imag(s_j)$ denotes reactive power.
Let $v,s \in \C^M$ be the stacked versions of the bus voltages and powers,
respectively. The admittance matrix of the grid is given as 
\begin{equation}
  Y_{jl} := \begin{cases} y_{jl} & \text{if } j \ne l, \\
  y^{\text{sh}}_{j}-\sum_{k=1,k\ne j}^M y_{jk} & \text{if } j = l.
  \end{cases}
\end{equation}
where $y_j^{\text{sh}} \in \C$ are shunt admittances. The Kirchhoff equations
for the system can hence be written in matrix form: 
\begin{equation}
  \label{eqn:kirchhoff}
  \diag(v)\bar Y \bar v = s,
\end{equation}
where $\bar \cdot$ describes the (element-wise) complex conjugate.
Let $e_k$ denote the $k$-th unit vector with appropriate dimension. Let
$(\cdot)^r := \real(\cdot), (\cdot)^q := \imag(\cdot)$ and let $r_k, q_k$ be
the $k$-th rows of $\real(Y)$ and $\imag(Y)$, respectively. 
For vectors $a \in \C^n$, define 
\begin{equation}
  J(a) := \left\{ k \in \{1,\ldots,n\} \Big| a_k \ne 0 \right\},
\end{equation}
and for matrices $A \in \C^{n \times n}$, let 
\begin{equation}
  \begin{aligned}
    J(A) :=& \;\Big\{ k \in \{1,\ldots,n\} \Big| \\
           & \;\exists j \in \{1,\ldots,n\} : A_{kj} \ne 0 \text{ or } A_{jk} \ne 0
    \Big\}.
  \end{aligned}
\end{equation}
This means $J(\cdot)$ returns the indexes of rows and columns with at least one
nonzero entry. We then use the notation $A_B$ to denote a version of $A$ with
only the rows and columns with indexes from a given set $B$.

\subsection{Operational constraints}
\label{ssec:opcon}
The constraints represent limits introduced by the system operator are either due
to regulations or to avoid damage to the system. Firstly, the voltage magnitude
has to be within a fixed interval for each bus $j$:
\begin{equation}
  \label{eqn:vlim}
  v_{\min,j} \le \left| v_j \right| \le v_{\max,j}.
\end{equation}
These limits are important for distribution grids, since the assumption of
low-resistance lines commonly made in transmission grids does not hold.
This means there can be significant discrepancies in the voltages between
two endpoints of a line. Additionally, one of the main problems faced by DSOs
are voltage constraint violations due to local renewable power in-feeds.
Finally, the current through each line $(j,k)$ is limited for thermal reasons:
\begin{equation}
  \label{eqn:llim}
  |y_{jl}||v_j - v_l| \le i_{\max,jl},
\end{equation}
The limits in~\eqref{eqn:vlim} and~\eqref{eqn:llim} together will hereafter be
referred to as the operational constraints for the power grid.
The DSO action space is modeled as an interval of active and reactive power for
each bus $j$:
\begin{equation}
  \label{eqn:slim}
  \begin{aligned}
    p_{\min,j} &\le \real( s_j) \le p_{\max,j}, \\
    q_{\min,j} &\le \imag( s_j) \le q_{\max,j}.
  \end{aligned}
\end{equation}
For buses at which the DSO cannot intervene, the upper and lower limits
in~\eqref{eqn:slim} are equal. Let $(s^0, v^0)$ be an operating point of the
power grid that represents the state of the distribution grid without any DSO
intervention. If this point satisfies all operational
constraints~\eqref{eqn:vlim} and~\eqref{eqn:llim}, no DSO intervention is
required. 
Otherwise, some limits are violated and the task of the DSO is then to find a
point $(s,v)$ that satisfies all operational constraints, but also lies within
its action space~\eqref{eqn:slim}. 

\subsection{DSO optimization problem}
The penalization for introduced deviations to power setpoints is modeled linearly
here, while voltage deviations are interpreted as an effect of changing powers
without a direct cost. This is the case for example in Germany~\cite{eeg2014}.
Even though the DSO is not run for profit, its operational cost has to be
covered by the power consumers. It is therefore advisable to perform a social
welfare optimization for least cost: 
\begin{subequations}
  \label{eqn:opf0}
\begin{align}
  \minim_{\cve s \in \C^M, \cve v \in \C^M} &\;\; 
    \|\real(\cve s-\cve s_0)\|_1 + \|\imag(\cve s-\cve s_0)\|_1 
    \label{eqn:opf0_cost}\\
    \st &\;\; \diag(\cve v) \jcve Y \jcve v = \cve s, \label{eqn:opf0_kirch}\\
      &\;\; v_{\min,k} \le \left| v_k \right| 
        \le v_{\max,k}, \label{eqn:opf0_vlim} \\
      &\;\; p_{\min} \le \real( s) \le p_{\max}, \label{eqn:opf0_plim} \\
      &\;\; q_{\min} \le \imag( s) \le q_{\max}, \label{eqn:opf0_qlim} \\
      &\;\; |y_{jl}||v_j - v_l| \le i_{\max,(j,l)}, \label{eqn:opf0_line} \\
      &\;\; k \in \mathcal M,\quad (j,l) \in \mathcal E,
\end{align}
\end{subequations}
where the $1$-norm cost function is proportional to the monetary cost for the
power deviations the DSO introduces. For renewable in-feed curtailment, this
situation is commonplace in some European countries, where the operator is
typically required by law to pay the nominal price for available power,
regardless of whether it is used or curtailed. Since this is the most relevant
case here, the assumption is made that all costs are of this structure. 
However, the general framework presented in this work can be extended to use
any convex cost function. Problem~\eqref{eqn:opf0} will hereafter be
referred to as the OPF problem.
It is non-convex due to the quadratic Kirchhoff
equalities~\eqref{eqn:opf0_kirch} as well as the lower voltage magnitude
bounds~\eqref{eqn:opf0_vlim}.

\subsection{Difference-of-convex-functions (DC) programming}
The method used in this work for solving problem~\eqref{eqn:opf0} is called
difference-of-convex-functions (DC\footnote{Not to be confused with the
abbreviation ``DC'' for direct current, and the related approximations 
of the AC-OPF problem.})
programming. This section outlines the algorithm and presents some existing
related theoretical results. 
DC programming is a class of algorithms for solving problems of the form 
\begin{equation}
  \label{eqn:dc_gen0}
\begin{aligned}
  \minim_{x} & \;\; g_0(x) - h_0(x) \\ 
  \st &\;\; g_i(x) - h_i(x) \le 0, 
\end{aligned}
\end{equation}
where $i \in \{1,\ldots,m\}$ and the $g_i, h_i$ are convex, subdifferentiable
functions. This method was historically used for optimization
problems involving piecewise affine functions.  However, a wide range of
problems can be formulated as~\eqref{eqn:dc_gen0}, including all convex
optimization problems, optimization problems with binary variables, quadratic
equality constraints and higher-order polynomial constraints.  A recent survey
of the method and related theory is given in~\cite{An2005}, and~\cite{LeThi2014}
presents the basic algorithm, which is also given in Algorithm~\ref{alg:dca}
for completeness. 
The main idea of the algorithm is to solve a sequence of convex problems
obtained by linearizing the \emph{concave} parts of the constraints and
objective:
\begin{equation}
  \label{eqn:dc_gen_inner0}
\begin{aligned}
  \minim_{x} 
  &\;\; g_0(x) - \left[ h_0(\tilde x) + \nabla h_0(\tilde x) (x-\tilde x) \right] \\
      \st &\;\; g_i(x) - \left[ h_i(\tilde x) + 
  \nabla h_i(\tilde x) (x-\tilde x) \right] \le 0.
\end{aligned}
\end{equation}
The optimizer $x^*$ of~\eqref{eqn:dc_gen_inner0} is then used as the next point
of convexification $\tilde x$, and the process is repeated until convergence is 
reached. The feasible set of~\eqref{eqn:dc_gen_inner0} is a convex inner
approximation of that of~\eqref{eqn:dc_gen0}. This means
that~\eqref{eqn:dc_gen_inner0} is not necessarily feasible, even if the
original non-convex problem is. This is circumvented in the algorithm using a
penalty reformulation:
\begin{equation}
  \label{eqn:dc_gen_inner1}
\begin{aligned}
  \minim_{x,t} 
  &\;\; g_0(x) - 
  \left[ h_0(\tilde x) + \nabla h_0(\tilde x) (x-\tilde x) \right] + \beta^k t \\
      \st &\;\; g_i(x) - \left[ h_i(\tilde x) + 
  \nabla h_i(\tilde x) (x-\tilde x) \right] \le t, \\
  &\;\; t \ge 0,
\end{aligned}
\end{equation}
where $\beta^k \in \R_+$ is a penalty weight parameter that is updated after
each convexification using the rule in Algorithm~\ref{alg:dca}. 

Note the similarity of this scheme to other sequential convex programming
methods, most notably SQP and SLP. The key difference to those methods is that
the convex parts $g_i$ are retained in their original form, which yields
tighter approximations in the sequence of convex problems solved. This means in
particular that if the original problem had a bounded feasible set, all issues
of possible unboundedness that arise in SLP~\cite{Bazaraa2013} are avoided and
there is no need for trust region approaches and their associated performance
penalty.

\begin{figure}
  \vspace{0.2cm}
  \begin{algorithmic}[1]
    \small
    \State Let $x^{0}$ initial guess, $\beta^0, \delta_1,\delta_2 > 0$
    parameters, $\epsilon_x, \epsilon_t > 0$ tolerances
    \label{algstep:termination}
    \While{Not converged}
    \State $x^{k+1},\lambda^{k+1},t^*\gets $ Solution of \eqref{eqn:dc_gen_inner1} 
      \label{algstep:inner}
    \If{$\| x^{k+1} - x^k\| \le \epsilon_x$ and $t \le \epsilon_t$}
      \State Terminate, converged to local optimality.
      \EndIf \vspace{0.1cm}
      \State $r^k \gets \min\left\{\left(\|x^{k+1} - x^{k}\|_2\right)^{-1},
    \|\lambda^{k+1}\|_1 + \delta_1 \right\}$ \vspace{0.1cm}
    \State $\beta^{k+1} \gets  \begin{cases} \beta^k & \text{ if } \beta^k \ge r^k \\
     \beta_k + \delta_2 &  \text{ if }\beta^k < r^k \end{cases}$
    \EndWhile
  \end{algorithmic}
  \caption{Difference of convex functions algorithm from~\cite{LeThi2014},
    modified to use practical stopping criteria. The $\lambda^{k+1}$
    computed in Step~\ref{algstep:inner} is the vector of dual multipliers of
    the constraints of the inner problem. }
  \label{alg:dca}
\end{figure}
It is shown in~\cite{LeThi2014} that the algorithm presented here globally
converges to a KKT point of~\eqref{eqn:dc_gen0} with a linear rate, provided
one exists and standard constraint qualifications are satisfied. It is also
shown that the sequence of optimal values of the
approximations~\eqref{eqn:dc_gen_inner1} is monotonically decreasing.  This
holds for all initial choices of the algorithm parameters, which means no
a priori bound on the size of the penalty parameter $\beta_k$ is required. All
results also hold if, in addition to the constraints in~\eqref{eqn:dc_gen0}, a
constraint
\begin{equation}
  \label{eqn:adddccon}
x \in \mathcal C,
\end{equation}
for some convex, closed set $\mathcal C$ is added. The algorithms are then
simply modified to include the constraint~\eqref{eqn:adddccon} in each of the
convex approximations. It is worth noting that the choice of $g_i$ and $h_i$
is not unique for a given problem. The authors in~\cite{LeThi2014} make no
theoretical statements on the impact of the choice of the $g_i$ and $h_i$ on
the convergence speed. However, numerical experiments show that the choice does
have a strong impact on the number of iterations required. We will discuss an
effective technique for choosing the functions $g_i,h_i$ for the problem at
hand in Section~\ref{ssec:appofdc}.

\section{Reformulation of OPF as Difference-of-Convex-functions problem}
\label{sec:reform}
In this section, the DSO OPF problem~\eqref{eqn:opf0} is reformulated as a
QCLP, and an efficient way of computing the splits of the non-convex functions
into differences of convex functions is presented. These splits result in a
special structure of the convex sub-problems. We then show in
Section~\ref{sec:effinner} how to solve these sub-problems efficiently. 

\subsection{Reformulation as QCLP}
\label{ssec:reform_qclp}
As already shown in~\cite{Low2014}, OPF problems with linear cost functions
can be recast as non-convex quadratically constrained linear programs. A
similar technique will be applied here. First, let $s_0 \in \C^M, v_0 \in \C^M$
be the power and voltage vectors the system is operating at without any DSO
intervention. 
We now introduce the difference in voltages introduced by
the DSO as follows:
\begin{equation} \label{eqn:dv_def} v := v_0 + \Delta v \;\in \C^{M}, \end{equation}
where $\Delta v \in \C^M$ is the change from the starting point and $v$ is the
resulting voltage vector. The resulting change of powers $\Delta s \in \C^M$ can be
computed using the Kirchhoff equations~\eqref{eqn:opf0_kirch}:
\[ \Delta s = \diag(v_0)\bar Y\bar{\Delta v} + \diag(\Delta v)\bar Y \bar v 
+ \diag(\Delta v)\bar Y \bar{\Delta v}. \]
Define now $Y^{(k)}$ as a version of $Y$ with all but the $k$-th row set to
0. 
After some reformulation, we can write
\begin{subequations} 
  \label{eqn:pow_from_volt}
  \begin{align} 
    \Big(\real (\Delta s) \Big)_k &= z^TH_{r,k}z + h_{r,k}^T z, \\
    \Big(\imag (\Delta s) \Big)_k &= z^TH_{q,k}z + h_{q,k}^T z, 
  \end{align}
\end{subequations}
with $z := \bmb \real(\Delta v)^T & \imag(\Delta v)^T \bme^T \in \R^{2M}$, and
\begin{equation}
  \label{eqn:pfv_matrices}
  \begin{aligned}
    H_{r,k} &:= \bmb 
    \real(Y^{(k)}) & -\imag(Y^{(k)}) \\ 
    \imag(Y^{(k)}) & \real(Y^{(k)}) 
    \bme,  \\
    H_{q,k} &:= \bmb 
    -\imag(Y^{(k)}) & -\real(Y^{(k)}) \\ 
    \real(Y^{(k)}) & -\imag(Y^{(k)})  
    \bme. 
  \end{aligned}
\end{equation}
The linear parts in~\eqref{eqn:pow_from_volt} are given by 
\begin{align}
    \,&h_{r,k} :=\nonumber \\
    \,&\bmb 
    \Big( (v_0^r)_k r_k + (v_0^q)_k q_k\Big)^T + e_k \Big(r_k (v_0^r)_k + q_k(v_0^q)_k\Big) \\
    \Big( (v_0^q)_k r_k - (v_0^r)_k q_k\Big)^T + e_k \Big(q_k (v_0^r)_k + r_k(v_0^q)_k\Big)
      \bme,\nonumber \\
    \,&h_{q,k} := \\
    \,&\bmb 
    \Big( (v_0^q)_k r_k - (v_0^r)_k q_k\Big)^T - e_k \Big(q_k (v_0^r)_k + r_k(v_0^q)_k\Big) \\
    \Big( -(v_0^q)_k q_k - (v_0^r)_k r_k\Big)^T + e_k \Big(r_k (v_0^r)_k - r_k(v_0^q)_k\Big)
      \bme.\nonumber 
\end{align} 
Equations~\eqref{eqn:pow_from_volt} can now be used to express the constraints
on powers given in~\eqref{eqn:opf0_plim}--\eqref{eqn:opf0_qlim} as constraints
on $\Delta v$. Using~\eqref{eqn:dv_def}, the constraints~\eqref{eqn:opf0_vlim}
and~\eqref{eqn:opf0_line} can also be expressed in $\Delta v$. Finally,
problem~\eqref{eqn:opf0} can be rewritten entirely in the variable $z$:
\begin{subequations}
  \label{eqn:opf1}
  \begin{align}
    \minim_{z \in \R^{2M}} &\;\; \sum_{k=1}^M \big|z^TH_{r,k}z + h_{r,k}^Tz \big| + 
    \big|z^TH_{q,k}z + h_{q,k}^Tz \big| \label{eqn:opf1_cost} \\
    \st & \;\; z^TQ_iz + q_i^Tz + \gamma_i \le 0 \label{eqn:opf1_con}, \\
        & \;\; i \in \{1,\ldots,K\},\nonumber
  \end{align}
\end{subequations}
where $K := 6M+L$.  The constraints~\eqref{eqn:opf1_con} are reformulations of
the original constraints~\eqref{eqn:opf0_vlim}--\eqref{eqn:opf0_line}.  The
structures of the $Q_i$ are of particular importance in later sections, which
is why they are given here.
\begin{figure}
  \includegraphics[width=\columnwidth]{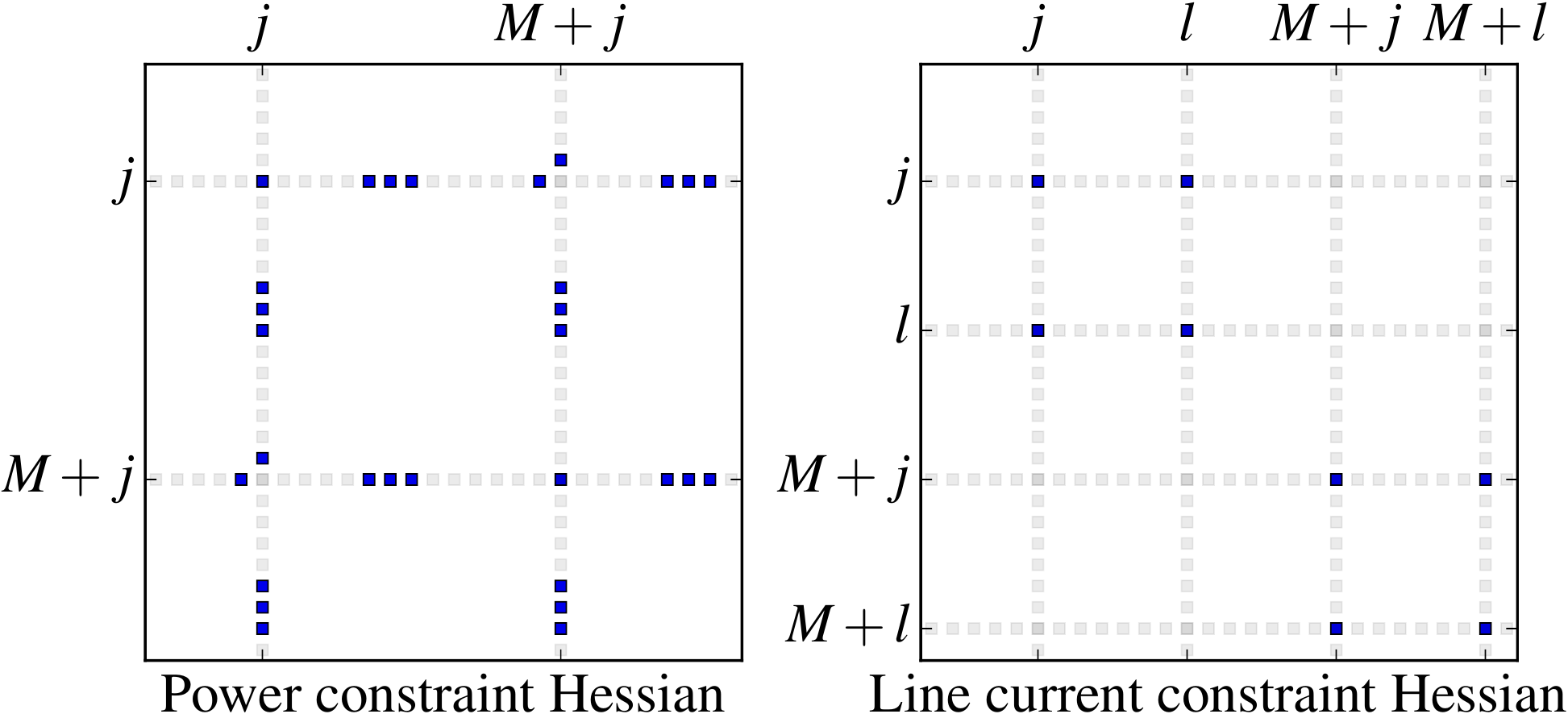}
  \caption{Sparsity pattern of example power constraint and line 
    constraint Hessian matrices $Q$. In this case, vertex $j$ has 4 neighbors.
    Note that the power constraint matrix is shown in symmetric form as defined
    in~\eqref{eqn:symmh}.}
  \label{fig:sparsevis1}
\end{figure}
\begin{enumerate}[(i)]
  \item The matrix $Q$ of power constraints~\eqref{eqn:opf0_plim}
    and~\eqref{eqn:opf0_qlim} are either the matrices
    $H_{r,k}$ or $H_{q,k}$ or negative versions thereof. 
  \item For the voltage bounds~\eqref{eqn:opf0_vlim}, $Q$ has $1$ (for upper
    bounds) or $-1$ (for lower bounds) on the $j$-th and $(M+j)$-th entries on
    the diagonal, and 0 everywhere else. 
  \item The matrix $Q$ of line constraints~\eqref{eqn:opf0_line}  have $1$ in
    positions 
    \[ (j,j),\; (l,l),\; (M+j,M+j),\;(M+l,M+l), \]
    of the diagonal and $-1$ in positions 
    \[ (j,l),\; (l,j),\; (M+j,M+l),\;(M+j,M+l). \]
\end{enumerate}
Note that the matrices from~\eqref{eqn:pfv_matrices} are not symmetric,
but they can be trivially made symmetric without changing the value of the
constraints in~\eqref{eqn:opf1_con}. We hence define the symmetric versions
\begin{equation}
  \label{eqn:symmh}
  \hat H_{r,k} := \frac{H_{r,k} + H_{r,k}^T}{2}, \qquad
  \hat H_{q,k} := \frac{H_{q,k} + H_{q,k}^T}{2}.
\end{equation}
A visualization of the described sparsity patterns is given in
Figure~\ref{fig:sparsevis1}.
Since the cost function~\eqref{eqn:opf1_cost} is inconvenient due to its
non-smoothness, a standard 1-norm reformulation with additional slack variables
$u \in \R^{2M}$ is performed.
Defining $x := \bmb z^T & u^T \bme^T$, problem~\eqref{eqn:opf1} can be written
as a standard QCLP:
\begin{subequations}
  \label{eqn:qclp}
\begin{align}
  \minim_{x} & \;\;  c^Tx  \label{eqn:qclp_cost} \\ 
  \st &\;\; x^TP_ix + p_i^Tx + \omega_i \le 0  \label{eqn:qclp_cons},\\
      &\;\; i \in \{1,\ldots,10M+L\},\nonumber \\
      &\;\; x_j \ge 0, \\
      &\;\; j \in \{2M+1,\ldots,4M\}, \nonumber
\end{align}
\end{subequations}
for appropriate $c,P_i, p_i, \omega_i$. The structure of the matrices $P_i$ 
is given by
\begin{equation}
  P_i = \bmb * & 0^{2M\times 2M} \\ 0^{2M\times 2M} & 0^{2M\times 2M} \bme
  \in \R^{4M \times 4M},
\end{equation}
where the upper blocks denoted by $*$ have the same sparsity patterns as the
matrices from Problem~\eqref{eqn:opf1}. The vectors $p_i \in \R^{4M}$ are
versions of the linear parts $h_{r,k},h_{q,k},q_i$ from
Problem~\eqref{eqn:opf1} with $2M$ additional entries. These additional entries
correspond to the coefficients of the slack variables $u$, at most one of which
is involved in each constraint. 

\subsection{Application of DC programming}
\label{ssec:appofdc}
In order to apply DC programming to solve~\eqref{eqn:qclp},
both~\eqref{eqn:qclp_cost} and~\eqref{eqn:qclp_cons} have to be written as a
difference of two convex functions as described in~\eqref{eqn:dc_gen0}. We
call this procedure a ``DC split''.  Since~\eqref{eqn:qclp_cost} is linear, no
split has to be performed, we can just define $g_0(x) := c^Tx$ and $h_0(x) :=
0$. The constraints~\eqref{eqn:qclp_cons} on the other hand can be non-convex,
so they have to be separated. Note that for every symmetric indefinite matrix
$P$, there exist infinitely many pairs $P^+, P^- \succeq 0$ such that
\begin{equation}
  \label{eqn:qsplit}
  P = P^+ - P^-.
\end{equation}
As a consequence, problem~\eqref{eqn:qclp}
can be rewritten as 
\begin{subequations}
  \label{eqn:qclp2}
\begin{align}
  \minim_{x} & \;\;  c^Tx \label{eqn:qclp2_cost} \\ 
  \st &\;\; \big(x^TP^+_ix + p_i^Tx + \omega_i\big) - \big(x^TP^-_ix\big) \le 0, 
      \label{eqn:qclp2_cons}\\
      &\;\; i \in \{1,\ldots,K\}\nonumber,\\
      &\;\; x_j \ge 0, \\
      &\;\; j \in \{2M+1,\ldots,4M\}, \nonumber
\end{align}
\end{subequations}
which now has the form given in~\eqref{eqn:dc_gen0}. Therefore, the algorithm
from Figure~\ref{alg:dca} can directly be applied.

The existence of infinitely many splits of the $P$ matrices
from~\eqref{eqn:qclp} raises the question of optimal split selection. One
approach that both intuitively makes sense and has been effective in
experiments is to split the matrices such that the $P_i^-$ have small
eigenvalues. This split strategy leads to the curvature of the concave terms $-x^TP_i^-x$
being smaller and therefore the linearized approximation being closer to the original
non-convex term. One can also use the freedom in the splits to induce structure
in the Hessian matrix $P^+_i$ in order to simplify the convex problems to be
solved. For example, the structure imposed here is for the matrices $P^+_i$ to
be diagonal, making their inverses trivial to compute.

\subsection{Analytic eigenvalue computations}
Since there are a large number of constraints of the
type~\eqref{eqn:qclp_cons}, calculating splits using numerical eigenvalue
decompositions would be computationally prohibitive. Due to the structure of
the $P_i$, eigenvalues can be computed analytically using the method described
in this section.
Note first that for the indexes $i$ corresponding to voltage or line
constraints, the eigenvalues of $P_i$ are trivial to compute due to their
simple structure. For the power constraints, we use the following Lemma:

\begin{lemma}
  The eigenvalues of the matrices from~\eqref{eqn:symmh} are given by
  \begin{subequations}
    \begin{align}
      \eig(\hat H_{r,k}) &= \left\{\frac
      {\real(Y_{kk}) \pm \sqrt{\real(Y_{kk})^2 - 4\|Y^{(k)}\|_2^2} }{2}, 0
      \right\}, \\[0.3cm]
      \eig(\hat H_{q,k}) &= \left\{ \frac
    {-\imag(Y_{kk}) \pm \sqrt{\imag(Y_{kk})^2 - 4\|Y^{(k)}\|_2^2}}{2},0\right\}.
  \end{align}
  \end{subequations}
\end{lemma}
\begin{proof}
  The proof is shown for $\hat H_{r,k}$ only, since the proof for the $\hat
  H_{q,k}$ is identical.  Note first that the $\hat H_{r,k}$ have a blocked
  structure:
\begin{equation}
  \label{eqn:blocked}
  \hat H_{r,k} = \bmb A_k & B_k \\ -B_k & A_k \bme = 
  \bmb A_k & B_k \\ B_k^T & A_k \bme.
\end{equation}
For matrices of this form, the identity
\begin{equation}
  \label{eqn:eigident}
  \eig\left( A_k + \sqrt{-1}B_k \right) = 
  \eig\left(\bmb A_k & B_k \\ -B_k & A_k \bme\right),
\end{equation}
holds~\cite{Golub2012}.
Both $A_k$ and $B_k$ are permuted arrowhead matrices.  A
matrix $A \in \C^{m \times m}$ is called \emph{arrowhead} if it has a structure
\begin{equation}
  A = \bmb \alpha & a^T \\ b & D \bme,
\end{equation}
with $\alpha \in \C$, $a,b \in \C^{m-1}$ and $D = \diag(d) \in \C^{(m-1)
\times (m-1)}$ for some $d \in \C^{m-1}$. In case $D = 0$, it can easily be
shown~\cite{OLeary1990} that
\begin{equation}
  \label{eqn:arroweig}
  \eig(A) = \left\{\frac{\alpha \pm \sqrt{\alpha^2 + 4a^Tb}}{2}, 0 \right\}.
\end{equation}
Next, note that $A_k$ only has one non-zero row at the same index as $B_k$ has
its only non-zero row, and the same holds for their columns. This means that
$A_k+\sqrt{-1}B_k$ is also arrowhead and its eigenvalues are the same as those
of $\hat H_{r,k}$ due to~\eqref{eqn:eigident}. Substituting $A_k =
\real(Y^{(k)}), B_k := \imag(Y^{(k)})$ and applying~\eqref{eqn:arroweig} yields
the lemma.
\end{proof}

Note that the application of~\eqref{eqn:arroweig} is particularly simple for
the case here, since $\hat H^{r,k}$ is built from $Y^{(k)}$, which in turn only
has as many entries as bus $k$ has neighbors. Since power system graphs are
generally very sparse, this yields a significant reduction in computational
cost over even a Lanczos-based or other iterative approximation of eigenvalues,
let alone a standard exact computation.
    
\subsection{Sparse splits}
At this point, the eigenvalues of all the matrices $P_i$ can be computed
efficiently.  However, directly applying the split in~\eqref{eqn:qsplit} would
lead to a loss of sparsity. Due to the sparse graph structure of the grid
matrix $Y$, the expressions $x^TP_ix$ only involve a small subset of the
variables in $x$ (specifically, the local variables for a bus and the variables
of its neighbors). This section introduces an alternative split that both
conserves sparsity in the constraints and also makes the $P^+_i$ diagonal. This
structure will then make the solution of the convex subproblems of the
algorithm much simpler, as will be outlined in later sections. This is because
the $P_i^+$ are used in the quadratic parts of the convex subproblems, whereas
the $P_i^-$ only appear in their linear terms. 

Define a sparse, diagonal matrix $D_i$ as follows:
\begin{equation}
  (D_i)_{jj} = \begin{cases} 
    1, & \text{if $j \in J(P_i) \cup J(p_i)$}, \\ 
    0, & \text{otherwise.} 
  \end{cases}
\end{equation}
This matrix hence has ones only at the row and column indexes at which either
$P_i$ and $p_i$ also have nonzeros.
We then define the alternative split
\begin{equation}
  \label{eqn:sparsesplit}
  P_i := \alpha D_i - (\alpha D_i - P_i),
\end{equation}
where $\alpha$ is the absolute value of the largest eigenvalue of $P_i$.
This sparse split still guarantees positive definiteness of the split matrices
since it only shifts the non-zero eigenvalues. 

\section{Efficient solution of inner problems}
\label{sec:effinner}

The bottleneck of the DC algorithm is the solution of the convex approximation.
In this section, an dual projected gradient method that has an iteration
complexity linear in the problem size. With the splits~\eqref{eqn:sparsesplit}
applied, the problem to be solved at each DC iteration has the form
\begin{subequations}
\label{eqn:dcinner_fromalgo}
\begin{align}
  \minim_{x \in \R^{4M}, t} &\;\; c^Tx + \beta^kt \\
  \st & \;\; x^TP_i^+x + \hat p_i(\tilde x^k)^Tx + \hat \omega_i(\tilde x^k)  \le t,  \\
      & \;\; i \in \{1,\ldots,10M+L\}, \quad t \ge 0, \\
      &\;\; x_j \ge 0, \\
      &\;\; j \in \{2M+1,\ldots,4M\}, \nonumber
\end{align}
\end{subequations}
where $\tilde x^k$ is the current point around which a convex approximation is 
formed, and 
\begin{equation}
  \begin{aligned}
    \hat p_i(\tilde x^k) &:= (p_i-2P_i^-\tilde x^k), \\
  \hat \omega_i(\tilde x^k) &:= \omega_i +(\tilde x^k)^TP_i^-\tilde x^k.
  \end{aligned}
\end{equation}
General-purpose sparse convex second-order cone programming codes such as
ECOS~\cite{Domahidi2013}, GUROBI~\cite{GurobiOptimization2014} or
MOSEK~\cite{ApS2015} can be used to solve these problems. However, the
structure of the problem suggests that a specialized solver could lead to
increased performance.  Firstly, the constraint Hessians $P_i$ are diagonal and
sparse, and all nonzero entries have the same values. Additionally, $P_i$ and
$p_i$ have the same nonzero patterns for any given $i$. It was also
experimentally observed that the subsequent convex approximations are often
similar, which suggests a warm-startable method could be beneficial. This
section will present an approach based on accelerated dual gradient descent
that was used in this work.

\subsection{Projected gradient method}
A well-known algorithm for solving optimization problems of the form
\begin{subequations}
\begin{align}
  \label{eqn:gradexample}
  \minim_x & \;\; f(x) \\ 
   \st &\;\; x \in \mathcal C
\end{align}
\end{subequations}
is given by the iteration
\begin{equation}
  \label{eqn:projgradalgo}
  x^{(k+1)} = \operatorname{proj}_{\mathcal C}
    \left(x^{(k)} - \alpha \nabla f(x^{(x)}) \right)
\end{equation}
where $\alpha$ is a step size and $\operatorname{proj}_{\mathcal C}$ is the
Euclidean projection onto the set $\mathcal C$. If $\mathcal C$ and $f(x)$ are
convex, this algorithm converges to the global minimum
of~\eqref{eqn:gradexample}, given that the step size is small enough. The rate of
convergence depends highly on $f(x)$, and the method can be accelerated by
varying $\alpha$ (see for example~\cite{Nesterov1983,Richter2012}). In order for
this algorithm to be efficient, the projection should be a simple operation. In
the next section, a reformulation of problem~\eqref{eqn:dcinner_fromalgo} is
given that achieves the latter. 

\subsection{Box-constrained inner problem formulation}
The intersection of the constraints in~\eqref{eqn:dcinner_fromalgo} is not easy
to project onto, hence direct application of~\eqref{eqn:projgradalgo} is not
efficient. Using two reformulations, the problem will be recast as a minimization
of a smooth function subject to box constraints. The first step is a lifting into
a higher-dimensional variable space: We introduce variables $y_i := x_i^2$ and
change the penalty function from an $\infty$-norm to a $1$-norm (another
possible penalty function shown in~\cite{LeThi2014}). The inner problem to be
solved then becomes
\begin{subequations}
  \label{eqn:dcinner_lifted}
\begin{align}
  \minim_{x,y,t} & \;\; c^Tx + \beta^k 1^Tt \label{eqn:liftedcost}\\ 
      \st & \;\; Ax + By - b \le t, \label{eqn:con_liftlin} \\
          & \;\; \diag(x)x - y = 0, \label{eqn:con_diagxy}\\
          & \;\; t \ge 0, \label{eqn:con_tpos} \\
          &\;\; x_j \ge 0, \\
          &\;\; j \in \{2M+1,\ldots,4M\}, \nonumber
\end{align}
\end{subequations}
where 
\begin{equation}
  A := 
    \begin{bmatrix} \hat p_1(\tilde x^k)^T \\ p_2(\tilde x^k)^T \\ 
      \vdots \\ p_K(\tilde x^k)^T
    \end{bmatrix}, \;
  B := 
    \begin{bmatrix} \diag(P_1^+)^T \\ \diag(P_2^+)^T \\ 
      \vdots \\ \diag(P_K^+)^T
    \end{bmatrix}, \;
  b := 
    \begin{bmatrix} \hat \omega_1(\tilde x^k) \\ \omega_2(\tilde x^k) \\ 
      \vdots \\ \omega_K(\tilde x^k)
    \end{bmatrix},
\end{equation}
where $K := 10M+L$. 
We now make use of the following lemma to relax the constraints
in~\eqref{eqn:con_diagxy}:
\begin{lemma}
  \label{lem:relax}
  Consider a version (P) of~\eqref{eqn:dcinner_lifted} with the
  constraints~\eqref{eqn:con_diagxy} relaxed to 
  \begin{equation} \diag(x)x - y \le 0. \label{eqn:proofcon}\end{equation}
  For every minimizer of (P) with one or more of~\eqref{eqn:proofcon} inactive,
  a minimizer with equal cost function can be found which has all the
  constraints~\eqref{eqn:proofcon} active.
\end{lemma}
\begin{proof}
The lemma will be shown by construction: Assume a point
  $(x^\star,y^\star,t^\star)$ is optimal for~(P), but a
constraint in~\eqref{eqn:proofcon} is not active. Then the $y_i$
corresponding to that constraint can be decreased to make the constraint
active without any change to the cost function or constraint satisfaction.
The latter is due to all entries of $B$ being non-negative. 
\end{proof}

Lemma~\ref{lem:relax} implies that we can simply solve the relaxed version
of~\eqref{eqn:dcinner_lifted} and then recover an optimal solution for the
latter.
In a second step, the relaxed version of~\eqref{eqn:dcinner_lifted} will be
dualized to yield a box-constrained problem. With some additional reformulation
(see Appendix~\ref{ssec:app_dualderi}), the dual of~\eqref{eqn:dcinner_lifted}
can be written as
\begin{subequations}
  \label{eqn:dcinner_dual}
\begin{align}
  \minim_{\lambda} & \;\; \frac{1}{4} 
    \lambda^TC\left(\diag(D^T\lambda)^{-1}\right)C^T\lambda + d^T\lambda \label{eqn:dualcost} \\ 
      \st & \;\; 0 \le \lambda \le 1.
\end{align}
\end{subequations}
for $C,D,d$ as derived in the appendix. 
This problem can now readily be solved using the projected gradient method and
its accelerated variants: The cost function is a sum of ``quadratic over linear''
functions, which are convex and differentiable:
\[ \frac{\partial }{\partial \lambda} \left(\frac{(a^T\lambda)^2}{b^T\lambda} \right)
  = 2a\frac{a^T\lambda}{b^T\lambda}  - b \frac{(a^T\lambda)^2}{(b^T\lambda)^2}.
\]
In order to avoid numerical issues with the inverse of $\diag(D^T\lambda)$, a
term $\varepsilon y$ can be added to~\eqref{eqn:liftedcost} for a small
$\varepsilon > 0$. This leads to the inverse in~\eqref{eqn:dualcost} becoming
$\diag(D^T\lambda + 1\varepsilon)^{-1}$, which is well-defined for all $\lambda
\ge 0$. Since the Lipschitz constant of~\eqref{eqn:dualcost} is not easily 
derived, an adaptive backtracking line search was used to determine the step
size $\alpha$ taken in~\eqref{eqn:projgradalgo}. The initial guess for the step
size was chosen to be 2 times the step size taken in the previous iteration.
This allows the algorithm to adapt its initial guess to both growing as well as
shrinking step sizes. At the same time, no convergence guarantees are
lost since the line search is still performed at each iteration.  This
technique led to a significant reduction in the average number of line search
iterations, speeding up the overall algorithm substantially.

\subsection{Computational complexity}
In order to investigate how scalable the presented method is, it is worthwhile
to compute the iteration complexity of the inner solver based on the problem 
parameters. Let $d_{\max}$ be the maximum number of neighbors of any vertex in
the power system graph, and recall that the number of buses and lines are
denoted by $M$ and $L$, respectively. The rows of the matrices $C$ and $D$ 
ultimately come from the $Y^{(k)}$ and variations thereof. Each of them has
at most $2d_{\max}$ entries. Since the number of rows in $C$ and $D$ is $\mathcal
O(M+L)$, this translates to the number of entries being $\mathcal O((M+L)d_{\max})$.
All that is required for the cost function and gradient computations is
products of $C^T$ and $D^T$ with $\lambda$ as well as some vector operations.
The projection is an elementwise operation and therefore~$\mathcal O(M+L)$. 
In summary, an iteration of the inner solver has linear complexity in the 
size of the grid if $d_{\max}$ is assumed to only grow very weakly with system
size, which is true in all test cases available.

\section{Numerical results}
\label{sec:numres}
In this section, we present numerical results on the performance and behavior
of the proposed algorithm. In order to make the results comparable to other
work in the field, some of the tests will be conducted on the IEEE benchmark
test systems available in MATPOWER~\cite{Zimmerman2011,Fliscounakis2013}. For
the experiments, a standalone implementation of the proposed method was
created, which will be referred to as DQ-OPF. The implementation is a
single-threaded, library-free ANSI C code, compiled with GNU GCC. The test
computer had a Core i7-4600U dual-core CPU clocked at 2.1 GHz and 8 GB of
memory. The operating system used was Debian Linux.

\subsection{Algorithm behavior}
\begin{figure}
  \includegraphics[width=\columnwidth]{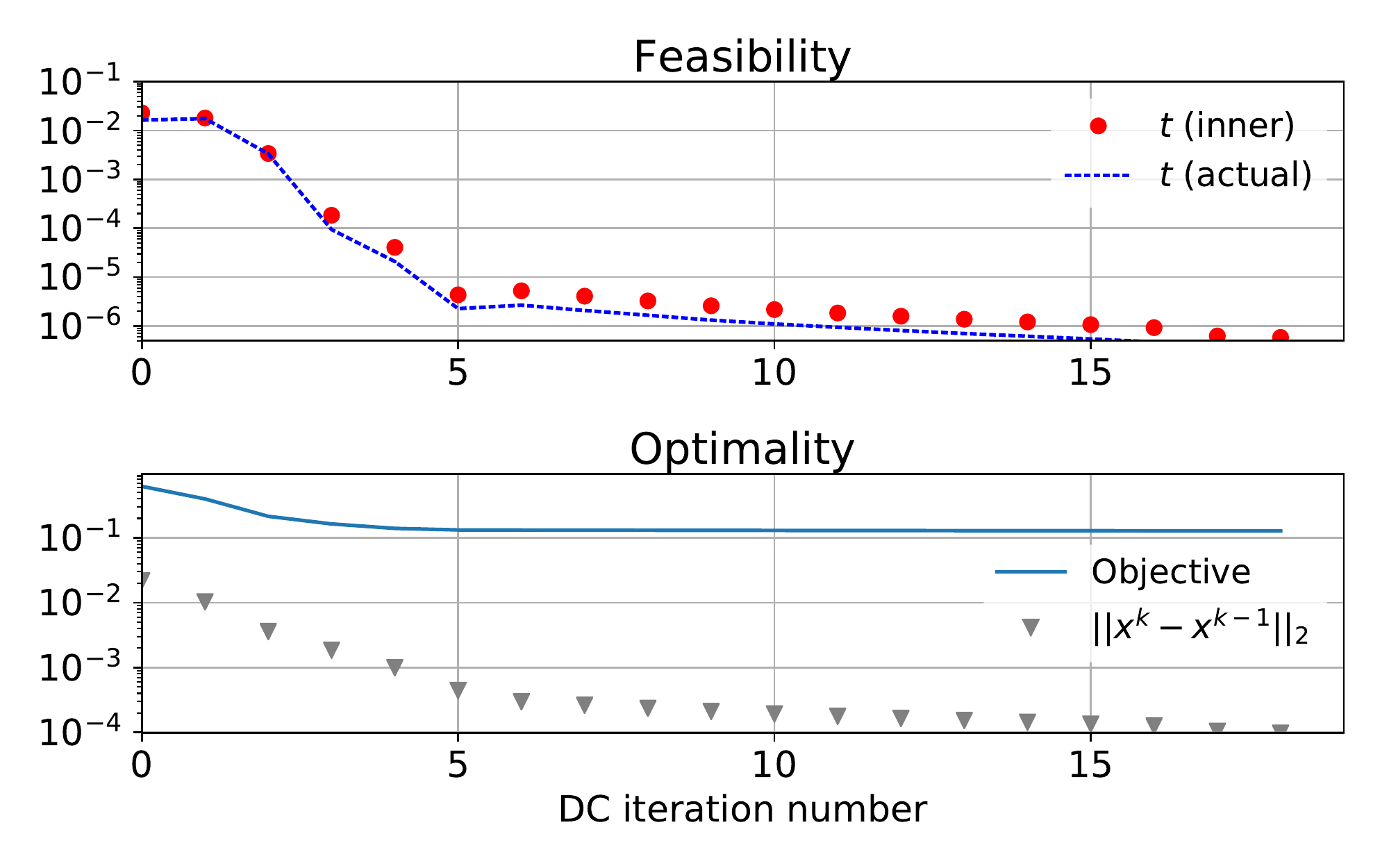}
  \caption{The DC method applied to the MATPOWER version of the IEEE 30-bus 
  grid, using dual projected gradient as the inner solver. The ``$t$ (inner)''
  line represents the maximum constraint violation of the convex approximation
  at that iteration, whereas the ``$t$ (actual)'' represents the maximum
  constraint violation of the original, non-convex problem. The lower subplot
  shows the true objective as well as the difference between subsequent
  iterates.} \label{fig:numres1}
\end{figure}
In the first set of results, the convergence behavior of the algorithm is
investigated. For these problems, a local optimum $(v^*,s^*)$ was found with
IPOPT~\cite{Waechter2006}. The entries of $v^*$ were then perturbed uniformly
and the corresponding perturbed powers were computed using the Kirchhoff
equations to yield a perturbed operating point $(\tilde v,\tilde s)$. 
The perturbation size was chosen to make the maximum constraint violation
about 100\%.  This was done in order to simulate the practical situation of the
power grid state being only slightly infeasible with respect to the operational
constraints, but respecting the Kirchhoff equations. The point $(\tilde
v,\tilde s)$ was used as starting point $(s^0,v^0)$ as defined in
Section~\ref{ssec:opcon}, and DQ-OPF started from there. 

An example solver run is shown in Figure~\ref{fig:numres1} with a termination
criterion of $\|x^k-x^{k-1}\|_2 \le 10^{-4}$. Within a small number of DC
iterations, the maximum constraint violation of the non-convex problem drops
below $10^{-4}$~per unit, which is well below $1\%$ relative accuracy. Note
also that 
the objective value does not improve
significantly past iteration 5. 

\begin{figure}
  \includegraphics[width=\columnwidth]{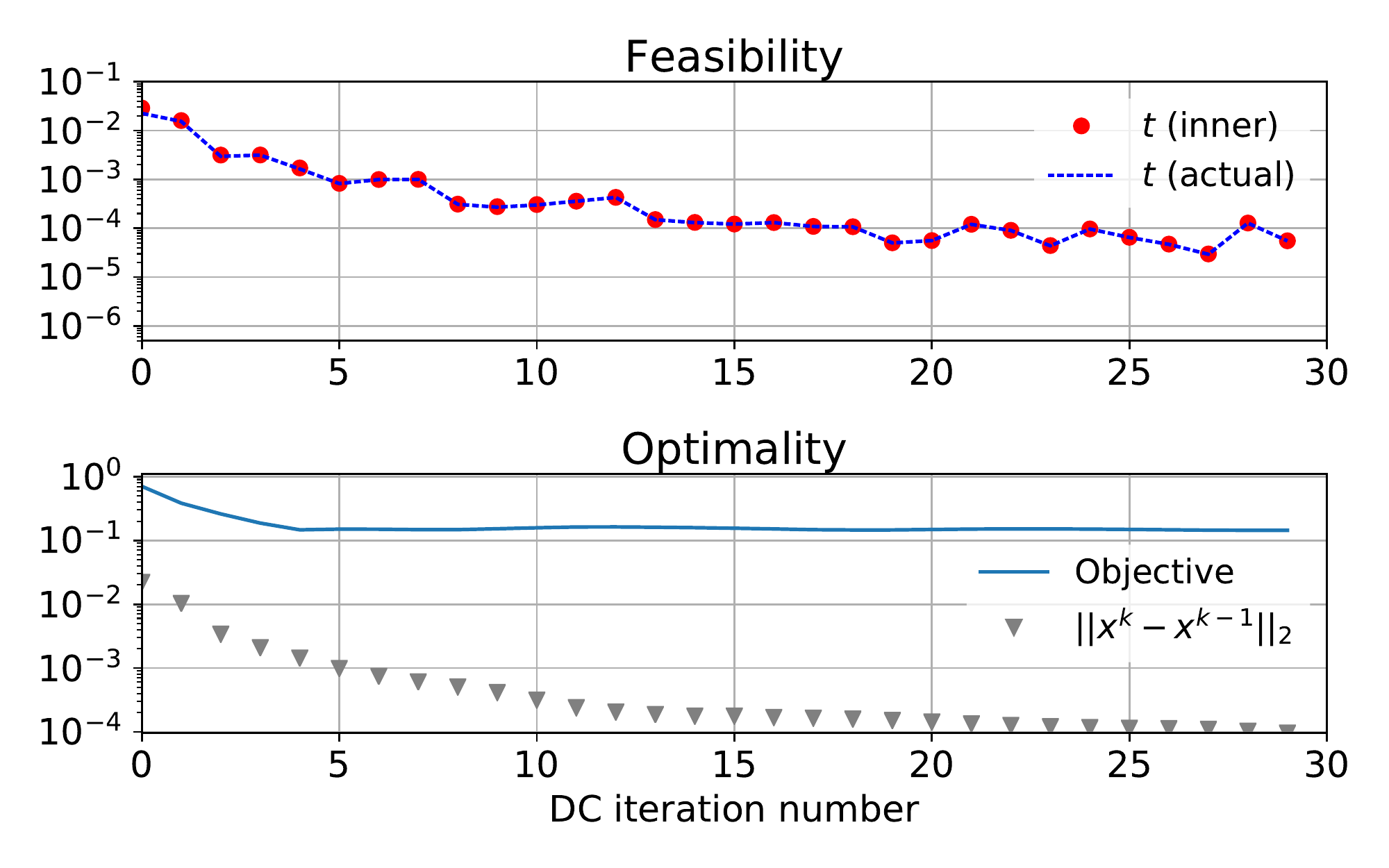}
  \caption{Solution of the same problem as in Figure~\ref{fig:numres1}, but 
    with the accuracy for the inner problem severely limited. The optimal 
    value of the inner problems no longer decreases monotonically, due to 
    the inaccurate inner solutions.}
  \label{fig:numres2}
\end{figure}

For the problem shown in Figure~\ref{fig:numres1}, the inner convex problems
were solved to high accuracy ($10^4$ inner iterations). In other sequential
convex programming methods, it is often observed that solving the intermediate
problems approximately can often be sufficient for
convergence~\cite{Heinkenschloss2002}. In order to investigate if this is also
the case with the methods presented here, the gradient solver iterations were
limited to 100 in the same problem as above, and the other parameters left
unchanged. The resulting run for the same problem is shown in
Figure~\ref{fig:numres2}. While the number of outer iterations required is
higher than before for the same accuracy, they are two orders of
magnitude cheaper computationally. Note that the objectives in
Figure~\ref{fig:numres1} and Figure~\ref{fig:numres2} converge to slightly
different values. This is due to the two solver runs converging to different
local optima.

\subsection{Performance}
\begin{table}
  \centering
  \renewcommand{\arraystretch}{1.3}
  \caption{Average time and objective for $1\%$ relative accuracy}
  \begin{tabular}{l   c c | c c}
    \hline
    & \multicolumn{2}{ c }{\textbf{IPOPT \& PARDISO}} 
      & \multicolumn{2}{ c }{\textbf{DQ-OPF}} \\
    & Time & Objective & Time & Objective \\
    6-bus    &  55 ms  &    0.9 &  2.2 ms &    0.9  \\
    9-bus    &  65 ms  &    1.9 &  2.3 ms &    1.5 \\
    14-bus   &  68 ms  &    1.3 &  3.5 ms &    1.0  \\
    30-bus   &  77 ms  &    1.5 &   10 ms &    1.3   \\
    39-bus   &  92 ms  &   12   &   23 ms &   11   \\
    57-bus   &  97 ms  &    2.9 &   25 ms &    1.6 \\
    118-bus  & 213 ms  &   17   &   76 ms &    3.7 \\
    2383-bus & 3.5 s   &   24   &   4.4 s &   10.4  \\
    2737-bus & 3.3 s   &   12   &   2.4 s &    7.6 \\
    3210-bus & 2.8 s   &   14   &   4.0 s &    9.2 \\
    9241-bus &  15 s   &   26   &    16 s &    6.8 \\
    \hline
  \end{tabular}
  \label{tbl:comparison}
\end{table}
In order to compare the implemented method to the state of the art, MATPOWER
test cases were used in conjunction with the 1-norm cost function, as described
in~\eqref{eqn:opf0}. Instead of solving the inner problems accurately as shown
in Figure~\ref{fig:numres1}, the inner solver was limited to 100--1000 iterations
depending on grid size, yielding the aforementioned calculation time
improvements. The DC solver parameters were tuned for one instance of the
problem and then reused across all runs. The solver was started at $(s^0,v^0)$.
As a reference, we used MATPOWER's IPOPT interface along with the parallel
PARDISO~\cite{Kuzmin2013,Schenk2008,Schenk2007} solver for linear systems.
Table~\ref{tbl:comparison} presents the results averaged over 100 runs with
random initial points created as in the previous experiment. In these
experiments, IPOPT was warm-started at the same point as DQ-OPF
using MATPOWERs warm-start functionality.
DQ-OPF is faster in many cases, with the speedups for the smaller
grids being substantial. For the larger grids, the run times are comparable to
the reference. Since no significant effort was put into optimizing the solver
for larger grids, further speedups can be expected in the proposed method 
through parallelization and more efficient code. For the largest grid, MATPOWER
ran into memory issues on the computer used. DQ-OPF requires a memory amount
linear in the problem size and hence had no such issues. Note also that IPOPT
was run with multi-threading enabled (2 threads) and the times shown are wall
clock, not CPU time.

Another observation is that the average objective values were consistently
smaller with the method used. This means the proposed method found local optima
with better objective values. A likely reason for this is the objective
function, which represents distance from the starting point $(s^0,v^0)$. 
The presented method tends to find local optima close to the point at which
it was started, whereas IPOPT (and interior-point methods in general) seem
to benefit less from warm-start information~\cite{John2008}.
This difference in objective values was made both when IPOPT was warm-started
as well as when the default settings (no warm-start) were used.

\subsection{Case study: Simulation experiment}
\begin{figure}[t]
  \includegraphics[width=\columnwidth]{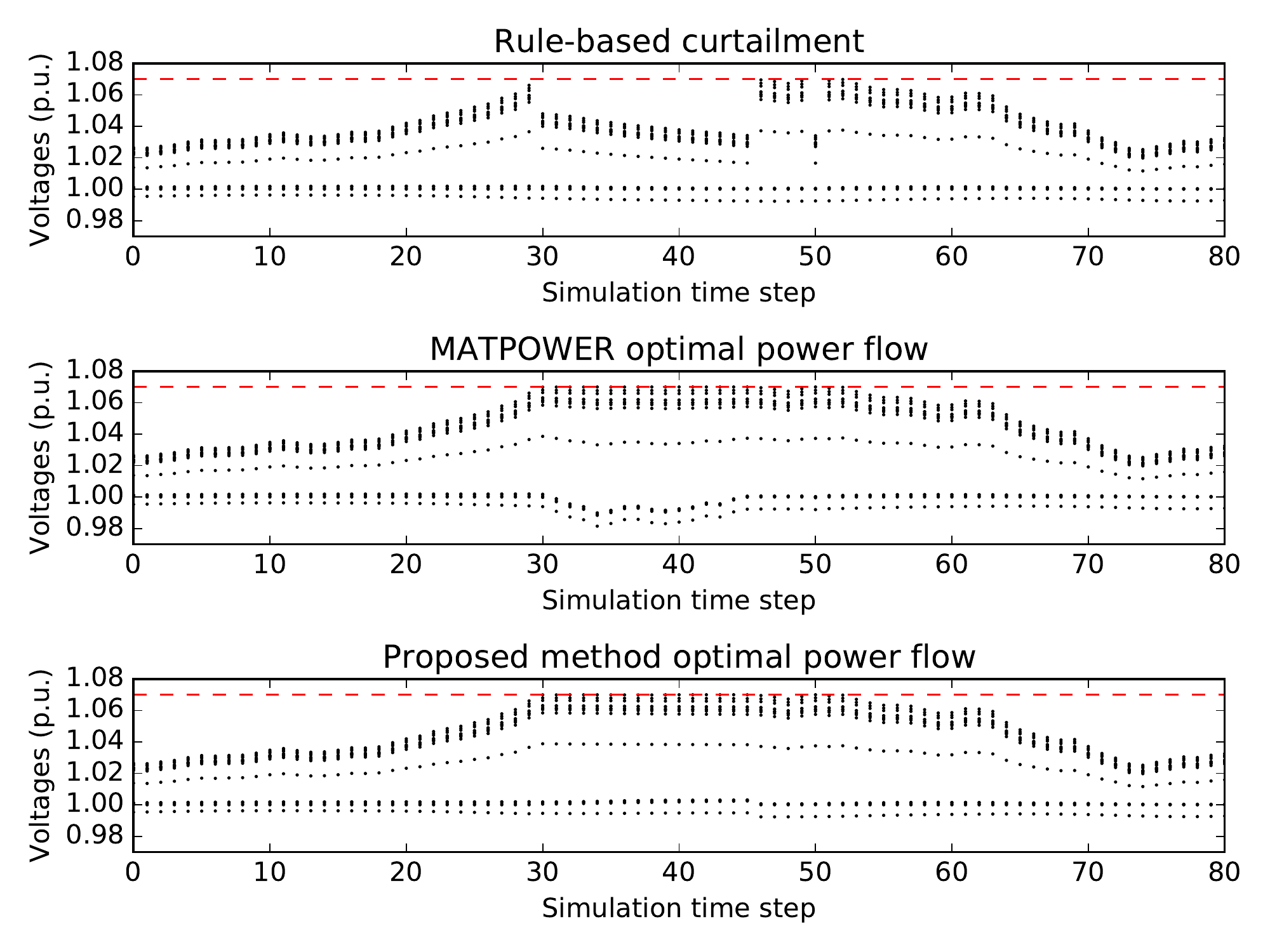}
  \caption{Voltage traces for all buses in the MV grid over the course of the
    simulation. Admissible limits were $[0.9,1.07]$ per unit. The OPF problem
    had to be solved at time instances 30 through 25 and 50.}
  \label{fig:numres3}
\end{figure}
\begin{figure}[t]
  \includegraphics[width=\columnwidth]{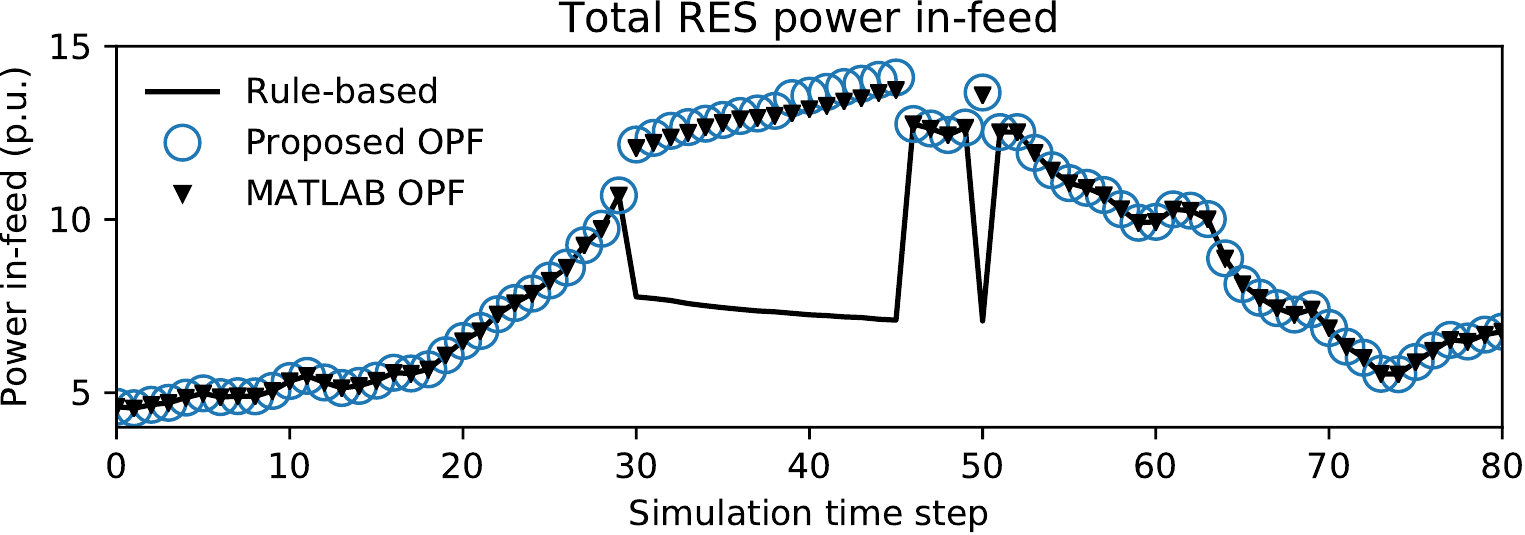}
  \caption{Total RES power in-feed for the MV grid over the simulation
    horizon. The MATPOWER and proposed OPF solutions are close, with the 
    proposed solution discarding slightly less renewable energy.}
  \label{fig:numres4}
\end{figure}
\begin{figure}[t]
  \includegraphics[width=\columnwidth]{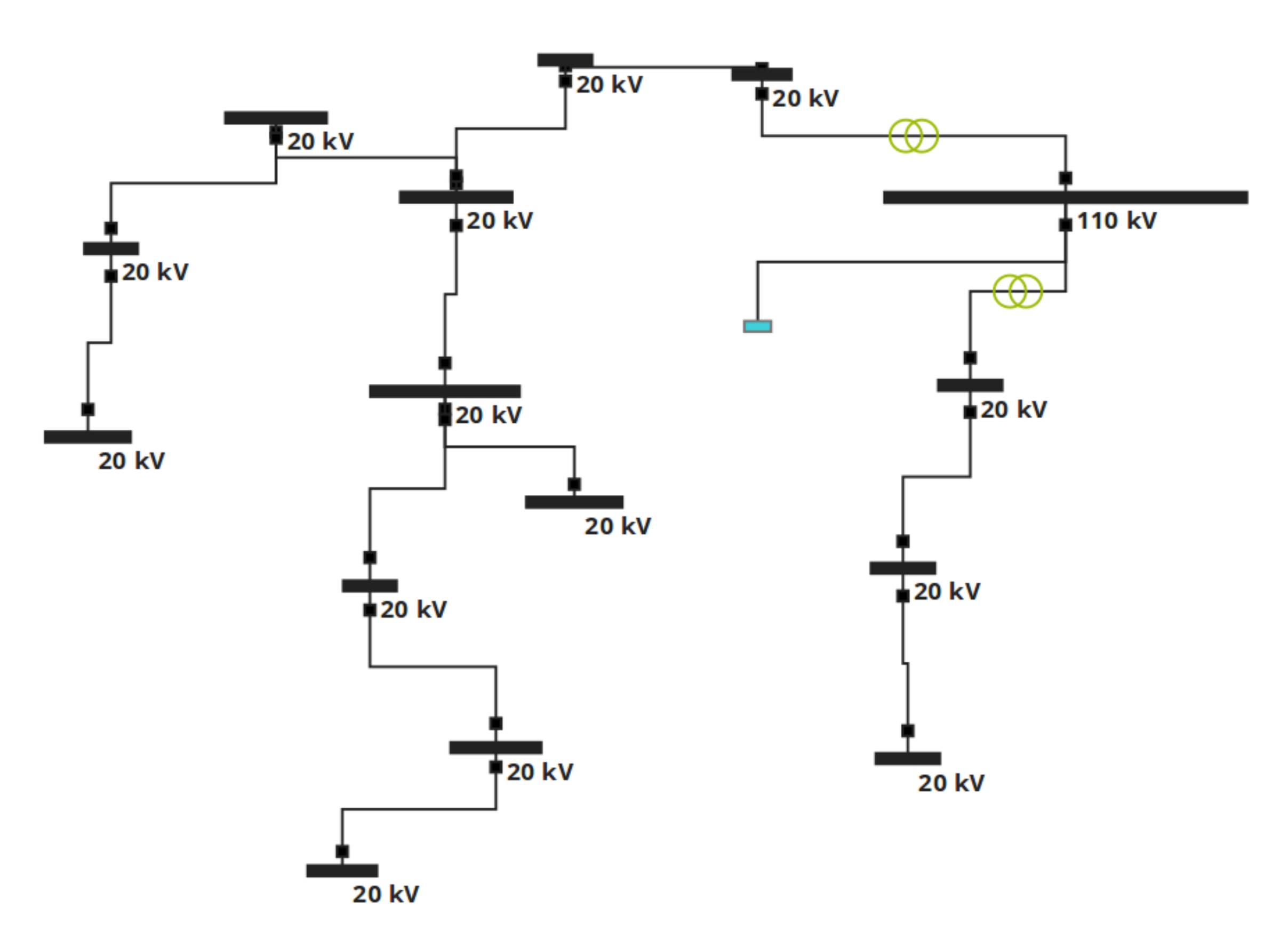}
  \caption{Rural MV grid used for the simulation experiment. Line admittances
    were chosen to be in range typically used in rural distribution grids.
    Note that while the grid here is radial, this is not a required assumption
    for the proposed method. The cyan bus in the middle is the slack bus,
  modeled here as a bus with no power limits.}
  \label{fig:mvruralgrid}
\end{figure}
\begin{figure}[t]
  \includegraphics[width=\columnwidth]{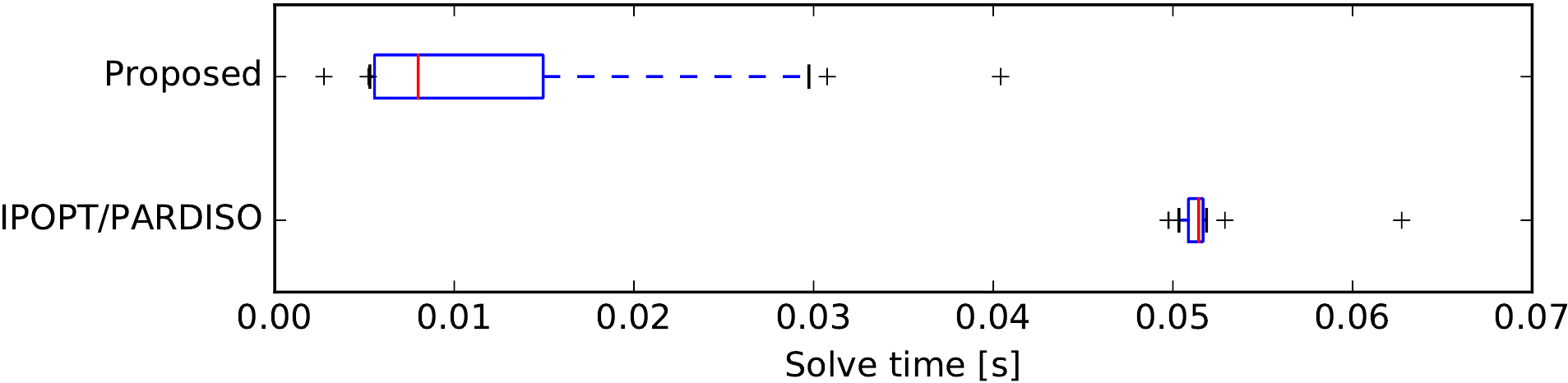}
  \caption{Box-plot of the distribution of solve times for the optimization
  problems solved in the simulation experiment. The boxes contain 50\% of the
  cases, the interval marked by the dashed lines contains 90\% of the cases.
  The plus signs mark outliers.}
  \label{fig:numres5}
\end{figure}
In order to demonstrate the effectiveness of warm-starting the presented
method, a power system time simulation experiment is presented in this section.
The experiment was run with the test grid shown in
Figure~\ref{fig:mvruralgrid}. Three different approaches to dealing with
voltage violations were tested: 
\begin{enumerate}[(i)]
  \item \emph{Rule-based curtailment}: In this control scheme, no optimization is
    run, instead the renewable in-feeds of the grid are simply curtailed down
    to a fixed fraction of their rated in-feed. This case reflects current
    industry practice.
  \item \emph{MATPOWER OPF-based curtailment}: In this scheme,
    problem~\eqref{eqn:opf0} is solved to local optimality with MATPOWER at
    its default settings. The 1-norm cost was implemented using MATPOWER's 
    piecewise affine cost function functionality.
  \item \emph{Proposed method OPF-based curtailment}: 
    Problem~\eqref{eqn:opf0} is solved to local optimality, but with the method
    presented in this paper. The solver is warm-started with the solution from
    the previous solve when available. The inner problems were solved with the
    presented dual gradient method, which was limited to 200 inner iterations.
\end{enumerate}
As a simulation environment, Adaptricity DPG.sim~\cite{KOCH} was used. At each
simulation time step, the operational
limits~\eqref{eqn:kirchhoff},~\eqref{eqn:vlim} and~\eqref{eqn:llim} were
checked. If any of them were violated, one of the approaches above was invoked.
The resulting power in-feed and voltage profiles for the different approaches
are shown in Figures~\ref{fig:numres3} and~\ref{fig:numres4}, respectively. As
can be seen in the uppermost subplots of the two figures, the profiles obtained
by using the rule-based curtailment controller have strong fluctuations due to
the controller intervening non-smoothly when violations are detected. Both the
voltage and power profiles are much smoother if the optimization-based
intervention solving~\eqref{eqn:opf0} is performed. Even though only local
optima are found both in MATPOWER and the presented method, these smoother
profiles were observed in all simulations. Additionally, even though the
different numerical approaches often yield different local minima, the
difference in cost function values is minor.
The distribution of solve times for the simulation is presented in
Figure~\ref{fig:numres5}. As can be seen, the average solve time of the
proposed method is only about 14\% of the state of the art. This directly
results in a speedup of up to factor 7 in the simulations.

Finally, due to the less severe interventions, much less curtailment is
required, resulting in a significant increase of renewable energy integrated.
The typical increase in RES in-feed is in the 20--40\% in yearly simulations,
but the specific value depends strongly on the grid topology and available
amount of renewable in-feed capacity.

\section{Conclusion}
\label{sec:conclusion}
This paper presented an alternative approach to dealing with over-voltage
problems in distribution grids with an OPF-based approach that leads to minimum
intervention by the DSO. Along with a formulation of the optimization problem,
a novel method to solve it to local optimality was presented that can be
warm-started and significantly outperforms current state-of-the-art
interior-point methods. The presented method can easily be extended to
other optimization problems involving AC power flow constraints. 

\section*{Acknowledgments}
\label{sec:acknowledgements}
This work was supported by the Swiss Commission for Technology and Innovation
(CTI), (Grant 16946.1 PFIW-IW). We also thank the team at Adaptricity (Stephan
Koch, Andreas Ulbig and Francesco Ferrucci) for providing the simulation
environment and valuable discussions in the area of power systems.

\appendix
\section{Appendix}
\subsection{Derivation of the dual problem}
\label{ssec:app_dualderi}
In this section, a detailed derivation of the step
from~\eqref{eqn:dcinner_lifted} to~\eqref{eqn:dcinner_dual} is given. First,
notice that the only non-zeros entries of $c$ in~\eqref{eqn:dcinner_lifted} are
those corresponding to the slack variables $u$ introduced in~\eqref{eqn:qclp}.
Moreover, the constraints corresponding to the cost function reformulation do 
not need penalties, since they are satisfied by construction. Recalling 
$x = \bmb z^T & u^T \bme^T$, we can rewrite~\eqref{eqn:dcinner_lifted} as
\begin{subequations}
  \label{eqn:dcinner_lifted_r1}
\begin{align}
  \minim_{x,y,u,t} & \;\; 1^Tu + \beta^k 1^Tt \label{eqn:liftedcost_r1}\\ 
      \st & \;\; A_1z + B_1y - b_1 \le t, \label{eqn:con_liftlin_r1} \\
          & \;\; A_2z + B_2y - b_2 \le u, \label{eqn:con_liftlin2_r1} \\
          & \;\; \diag(z)z - y \le 0, \label{eqn:con_diagxy_r1}\\
          & \;\; t \ge 0, \label{eqn:con_tpos_r1} \\
          & \;\; u \ge 0. \label{eqn:con_tpos2_r1}
\end{align}
\end{subequations}
In~\eqref{eqn:dcinner_lifted_r1}, constraints~\eqref{eqn:con_liftlin_r1}
contain all the actual constraints (originally~\eqref{eqn:opf1_con}),
whereas~\eqref{eqn:con_liftlin2_r1} contains all constraints resulting from the
cost function reformulation. In the 1-norm cost formulation, a trick has been
applied: Normally, a cost of $|w|$ for some variable $w$ would be replaced by
one slack variable $s$ and then a problem 
\[ 
\begin{aligned}
  \minim_{w,s} &\;\; s + \text{(other costs)} \\
   \st &\;\; w \le s,\; -w \le s,  \\
       &\;\; \text{(other constraints)},
\end{aligned}
\]
solved. An equivalent formulation to this is to introduce two slack variables
$s_1,s_2$ and solve 
\[ 
\begin{aligned}
  \minim_{w,s} &\;\; s_1 + s_2 + \text{(other costs)} \\
   \st &\;\; w \le s_1,\; -w \le s_2, \\
       &\;\; s_1 \ge 0, s_2 \ge 0, \\
       &\;\; \text{(other constraints)}.
\end{aligned}
\]
The equivalence is easily shown: One of $-w,w$ is always negative, leading to
one of $s_1,s_2$ becoming $0$, and the cost being equivalent to the more
standard formulation. Because this alternative formulation was used, one can
treat the $t$ and $u$ 
in~\eqref{eqn:dcinner_lifted_r1} the same and rewrite the latter once more as 
\begin{subequations}
  \label{eqn:dcinner_lifted_r2}
\begin{align}
  \minim_{x,y,t} & \;\; 1^Tt \label{eqn:liftedcost_r2}\\ 
      \st & \;\; Cz + Dy - d \le t, \label{eqn:con_liftlin_r2} \\
          & \;\; \diag(z)z - y \le 0, \label{eqn:con_diagxy_r2}\\
          & \;\; t \ge 0, \label{eqn:con_tpos_r2} \\
\end{align}
\end{subequations}
where $t$ now has a larger dimension than the $t$ in~\eqref{eqn:dcinner_lifted_r1}
and 
\[ C := \bmb \beta^k A_1 \\ A_2 \bme, \quad D := \bmb \beta^k B_1 \\ B_2 \bme,
  \quad d := \bmb \beta^k b_1 \\ b_2 \bme. 
\]
At this point, let $\lambda, \mu$ and $\gamma$ be the dual multipliers 
for the constraints~\eqref{eqn:con_liftlin_r2},~\eqref{eqn:con_diagxy_r2}
and~\eqref{eqn:con_tpos_r2}, respectively. The Lagrangian
of~\eqref{eqn:dcinner_lifted_r2} then becomes
\begin{equation}
  \begin{aligned}
    L(z,y,t,\lambda,\mu,\gamma) =\;\;& 1^Tt + \lambda^T(Cz + Dy - d - t) \\
    &+ \mu^T(\diag(z)z-y) + \gamma(-t).
  \end{aligned}
\end{equation}
Setting the partial derivatives to 0 yields
\begin{subequations}
\begin{align}
  1 - \lambda - \gamma &= 0, \label{eqn:dldt} \\
  C^T\lambda + 2\diag(\mu)z &= 0, \label{eqn:dldz} \\
  D^T\lambda - \mu &= 0. \label{eqn:dldy}
\end{align}
\end{subequations}
Equation~\eqref{eqn:dldz} implies that 
\[ z^*(\lambda,\mu) = -\frac{1}{2}\diag(\mu)^{-1}C^T\lambda. \]
The dual problem hence becomes
\begin{subequations}
  \label{eqn:dcinner_dual2}
\begin{align}
  \maxim_{x,y,t} & \;\; -\frac{1}{4}\lambda^TC\diag(\mu)^{-1}C^T\lambda - d^T\lambda \\
      \st & \;\; \lambda,\gamma,\mu \ge 0, \\
          &\;\;  1 - \lambda - \gamma = 0, \\
          &\;\; D^T\lambda - \mu = 0.
\end{align}
\end{subequations}
Upon closer inspection of~\eqref{eqn:dcinner_dual2}, it can be seen
that $\gamma$ and $\mu$ can be eliminated to yield
\begin{subequations}
  \label{eqn:dcinner_dual3}
\begin{align}
  \maxim_{x,y,t} & \;\; -\frac{1}{4}\lambda^TC\diag(D^T\lambda)^{-1}C^T\lambda - d^T\lambda \\
      \st & \;\; 0 \le \lambda \le 1, \label{eqn:dualcon1} \\
          &\;\; D^T\lambda \ge 0.\label{eqn:dualcon2}
\end{align}
\end{subequations}
Finally, since all entries of $D$ are non-negative, constraints~\eqref{eqn:dualcon1}
imply~\eqref{eqn:dualcon2}, and the latter can therefore be removed, resulting
in the formulation~\eqref{eqn:dcinner_dual} presented in the main text.

\bibliographystyle{plain}
\bibliography{merkli_etal_fast_ac_power_flow_optimization_using_dc_programming_draft2}
\end{document}